\documentclass[12pt, reqno]{amsart}
\usepackage{amssymb,dsfont}
\usepackage{eucal}
\usepackage{amsmath}
\usepackage{amscd}
\usepackage[dvips]{color}
\usepackage{multicol}
\usepackage[all]{xy}           
\usepackage{graphicx}
\usepackage{color}
\usepackage{colordvi}
\usepackage{xspace}
\usepackage{txfonts}
\usepackage{lscape}
\usepackage{tikz} 
\numberwithin{equation}{section}
\usepackage{ifpdf}
\ifpdf
 \usepackage[colorlinks,final,backref=page,hyperindex]{hyperref}
\else
  \usepackage[colorlinks,final,backref=page,hyperindex,hypertex]{hyperref}
\fi
\usepackage{tikz}
\usepackage[active]{srcltx}

\makeatletter

\input xy \xyoption{all} 

\def\n{\mathbb{N}}
\def\N{\mathbb{N}}
\def\z{\mathbb{Z}}
\def\Z{\mathbb{Z}}
\def\c{\mathbb{C}}

\def\dim{\hbox{dim}}

\newfont{\df}{eufm10}

\def\dim{\hbox{\rm dim}\,}

\def\Vir{\hbox{\rm Vir}}

\def\span{\hbox{\rm span}\,}

\numberwithin{equation}{section}

\def\BE{\begin{equation}}
\def\EE{\end{equation}}
\def\BA{\begin{eqnarray}}
\def\EA{\end{eqnarray}}

\title[Simple weight modules]{Classification of simple  modules with finite-dimensional weight spaces  for the $N=2$ Ramond algebra
}

\date{}

\author{Dong Liu}
\address{Department of Mathematics, Huzhou University, Zhejiang Huzhou, 313000, China}
\email{liudong@zjhu.edu.cn}

\author{Yufeng Pei}
\address{Department of Mathematics, Shanghai Normal University,
Shanghai, 200234, China} \email{pei@shnu.edu.cn}

\author{Limeng Xia$^\sharp$}
\address{Institute of Applied System Analysis, Jiangsu University, Jiangsu Zhenjiang, 212013,
China}\email{xialimeng@ujs.edu.cn}

\begin{document}

\begin{abstract}  In this paper, we classify all simple weight modules with finite-dimensional weight spaces over the $N=2$ Ramond algebra. Any such module $V$ is either a simple  highest weight module or a simple lowest weight module, or  a simple cuspidal module with ${\rm b}(V)\le 2$.
\end{abstract}

\newtheorem{theo}{Theorem}[section]
\newtheorem{defi}[theo]{Definition}
\newtheorem{lemm}[theo]{Lemma}
\newtheorem{coro}[theo]{Corollary}
\newtheorem{prop}[theo]{Proposition}
\newtheorem{rema}[theo]{Remark}

\thanks{$^\sharp$ the corresponding author}



\maketitle

\section{Introduction}
\setcounter{section}{1}\setcounter{theo}{0} \setcounter{equation}{0}

Representation theory of the Virasoro algebra   has produced a vast amount of applications  in  physics and mathematics.  It is well-known that the classification of simple  weight modules with finite-dimensional weight spaces over the Virasoro algebra, conjectured by V. Kac in \cite{Ka}, was completed by Mathieu in \cite{Ma} (see also \cite{KS, S1}). Later, simple weight  modules with finite-dimensional weight spaces for various extensions of the Virasoro algebra were studied (see \cite{S2} for the Weyl algebra, \cite{LvZ0, S4} for high rank Virasoro algebras, \cite{LvZ} for the twisted Heisenberg-Virasoro algebra, \cite{Liu} for the Schr\"{o}dinger-Virasoro algebra, \cite{LG} for the generalized Heisenberg-Virasoro algebra, etc.).  Recently, Billig and  Futorny  classified all simple weight modules with finite-dimensional weight spaces for the Lie algebra $W_n$ of vector fields on an $n$-dimensional torus (see \cite{BF} and the references therein).

Superconformal algebras may be viewed as
natural super-extensions of the Virasoro algebra and have been playing a fundamental role in  string theory and conformal field theory. In \cite{KL} Kac and van de Leur
gave a mathematically rigorous definition of a superconformal algebra as follows: (a) it is  a $\Z$-graded complex Lie
superalgebra, (b) it is graded simple, (c) the dimensions of its graded spaces are uniformly bounded. Furthermore, it contains the Virasoro algebra as a graded subalgebra.  In \cite{Ka1}, Kac classified all physical superconformal algebras:
namely, the $N = 0$ (Virasoro), $N = 1$ (Neveu-Schwarz), $N = 2, 3$ and $4$ superconformal
algebras, the superalgebra $W(2)$ of all vector fields on the $N=2$ supercircle, and
a new superalgebra $CK(6)$. Representation theory of superconformal algebras has been the subject of intensive study (see \cite{D, IK1,KL}, etc.).  It is a challenging problem to give complete classifications of simple weight modules with finite-dimensional weight spaces for superconformal algebras.  In \cite{CP}, all simple unitary weight modules with finite-dimensional weight spaces over the $N=1$ superconformal algebra were classified, which include highest and lowest weight modules. A complete classification for the $N=1$ superconformal algebra  was  given by Su in \cite{S0} (also see \cite{CLL, CL}). Such modules (also named graded modules here) for superconformal algebras were also studied in \cite{MZe}, where some cuspidal (also named uniformly bounded) modules were constructed. The problem of classifying such simple modules over superconformal algebras turns out to be more subtle.

It is well-known that the $N=2$ superconformal algebras fall into four
sectors: the Ramond sector, the Neveu-Schwarz sector, the
topological sector and the twisted sector.
Schwimmer and Seiberg demonstrated that the Neveu-Schwarz sector and Ramond sector are isomorphic in their work \cite{SS}. Dijkgraff et al. introduced the topological sector through topological twist in \cite{DVV}. Interestingly, the topological sector is isomorphic to the untwisted sectors (Neveu-Schwarz), which means there are only two distinct sectors in the $N=2$ superconformal algebras.
The main purpose of the present paper is to study simple weight modules with finite-dimensional weight spaces for the Ramond sector (see \cite{Io, SS}).

  Using the classification result on simple weight modules with finite-dimensional weight spaces  for the twisted Heisenberg-Virasoro algebra $\mathfrak t$ in \cite{LvZ} and some facts of Grassmann algebras,  we classify all simple cuspidal $\mathfrak q^\pm$-modules.
With this classification (Proposition \ref{pssn0}), we can classify all simple cuspidal $\mathfrak g$-modules (uniformly bounded modules) to the combinations of the two pieces of simple $\mathfrak q^\pm$-modules $U^\pm$ as the following diagram (see Proposition \ref{pn=2} for details).

\centerline{\xymatrix{
U^+\ar@{->}[r]\ar@{=}[d]^\cup&\mathfrak q^-U^+\ar@{->}[r]\ar@{<-}[d]^\cup&\cdots\ar@{->}[r]&(\mathfrak q^-)^{n}U^+\ar@{<-}[d]^\cup\\
(\mathfrak q^+)^{n}U^-\ar@{<-}[r]&(\mathfrak q^+)^{n-1}U^-\ar@{<-}[r]&\cdots\ar@{<-}[r]&U^-}}

By direct calculations we get a useful identity $G_{r_1}^-G_{r_2}^-G_{r_3}^-G_{s_1}^+G_{s_2}^+G_{s_3}^+U^-=0, \forall r_i, s_i\in\z, i=1,2,3$ for the $\mathfrak q^\pm$-modules $U^\pm$, and then obtain our main result for the $N=2$ Ramond algebra. The main theorem can be stated as follows.

\noindent{\bf Main Theorem} (see Theorem $\ref{main})$
{\it Let $V$ be a simple weight module with finite-dimensional weight spaces over the $N=2$ Ramond algebra. Then $V$ is a highest weight module, a lowest weight module, or ${\mathcal R}_{a, b, c}'$ for some $a, b, c\in\c$ $($up to parity-change$)$.}

This paper is arranged as follows. In Section 2, we recall some
notations and collect known facts about the $N=2$ Ramond algebra.
In Section 3, we classify all simple cuspidal modules for the subalgebra $\mathfrak q^\pm$ of the $N=2$ Ramond algebra. In Section 4, we classify all simple cuspidal modules of the $N=2$ Ramond algebra.  In Section 5, we show that all
simple non-cuspidal weight modules over the $N=2$ Ramond algebra which turns out to be highest (or lowest) weight modules, and then get the main result of this paper.

Throughout this paper, we shall use $\c, {\mathbb Q}$, $\z$, $\z^*$, $\z_+$ and $\mathbb N$ to denote the sets of the complex numbers, the rational numbers, the integers,  the non-zero integers, the positive integers and the nonnegative integers, respectively.
For convenience, all algebras (vector spaces) are based on the field $\c$,  all elements in superalgebras and modules are homogenous unless specified. We always denote by $U(L)$ the universal enveloping algebra of a given Lie (super)algebra $L$, and denote by $L^nV:=\span_\c\{x_1x_2\cdots x_n\cdot v\mid v\in V, x_i\in L, i=1,2,\cdots n\}$ for any positive integer $n$ and $L$-module $V$. A nonzero $v\in V$ is called a trivial $L$-vector if $Lv=0$.

\section{Preliminaries}
In this section, we collect the definitions about the  $N=2$ Ramond algebra and some  known facts for later use.

Let $L$ be a  Lie superalgebra. An $L$-module is a $\z_2$-graded vector space $V=V_{\bar0}\oplus V_{\bar 1}$ together with a bilinear map, $L\times V\to V$, denoted $(x,v)\mapsto xv$ such that
$$
x(yv)-(-1)^{|x||y|}y(xv)=[x,y]v
$$
for all $x,y\in L, v\in V$, and  $L_{\sigma} V_{\tau}\subseteq V_{\sigma+\tau}$ for all $\sigma, \tau\in\z_2$. Thus there is a parity change functor $\Pi$ on the category of $L$-modules, which interchanges
the $\z_2$-grading of a module.

The $N=2$ Ramond algebra ${\frak g}$ is a Lie superalgebra over $\c$ with a basis $\{L_m, H_m, G_m^\pm, C\mid m\in\z\}$ and the following relations:
\begin{eqnarray*}
&& [L_m,L_n]=(n-m)L_{n+m}+{1\over12}(n^3-n)C,\\
&&[L_m, H_n]=nH_{m+n},\\
&&[H_m,H_n]={1\over3}m\delta_{m+n,0}C,\\
&&\label{gd2}
[L_m,G_p^\pm]=(p-\frac{m}{2})G_{p+m}^\pm,\\
&&[H_m,G_p^\pm]=\pm G_{m+p}^\pm,\\
&&\label{gd3}[G_p^+,G_q^-]=-2L_{p+q}+(p-q)H_{p+q}+{1\over3}(p^2-{1\over4})\delta_{p+q,0}C,\\
&&[G_p^\pm,G_q^\pm]=0=[C,\frak g]
\end{eqnarray*}
for $m,n, p, q\in\z$.

The Lie superalgebra $\frak g$ is $\z$-graded by $(\frak g_i)_{\bar0}=\span_\c\{L_i, H_i\}+\delta_{i, 0}\c C$ and $(\frak g_i)_{\bar1}=\span_\c\{G_i^+, G_i^-\}$ for any $i\in\z$. Now we introduce some subalgebras of $\frak g$ as follows:
\begin{equation}\mathfrak t:=\frak g_{\bar0}=\span_\c\{L_n, H_n, C\mid n \in\z\},\label{deft}\end{equation}  which is called the twisted Heisenberg-Virasoro algebra (see \cite{ADKP, LJ, LvZ}, etc.).
\begin{equation}{\mathfrak q}^\pm:=\span_\c\{L_n, H_n, G_n^\pm, C\mid n\in\z\},\label{defq}\end{equation} which play a key role in our studies, and will be mainly considered in Section 3.

For any
${\frak g}$-module $V$ and $\lambda\in \c$, set
$V_{\lambda}:=\bigl\{v\in V\bigm|L_0v=\lambda v\bigr\}$, which is
generally called the weight space of $V$ corresponding to the weight
$\lambda$.
A ${\frak g}$-module $V$ is called a weight module if $V$ is the
sum of all its weight spaces.

For a weight module $V=V_{\bar0}+V_{\bar1}$, we define
\begin{equation}\hbox{Supp}(V):=\bigl\{\lambda\in \c \bigm|V_\lambda \neq
0\bigr\}.\end{equation}

If $V$ is a simple weight ${\frak g}$-module, then
there exists $\lambda\in\c$ such that ${\rm
Supp}(V)\subset\lambda+\z$. So $V=\sum_{i\in \z}V_{\lambda+i}$ is $\z$-graded (sometimes simply denoted by $V=\sum_{i\in \z}V_{i}$).
A weight ${\frak g}$-module $V=\sum V_{\lambda+i}$ is called {\it
quasi-finite} if all $V_{\lambda+i}$ are finite-dimensional.
If, in addition, there exists a positive integer $p$ such that \BE \dim (V_{\lambda+i})_{\tau}\le p,\ \forall i\in\z,\ \forall \tau\in\z_2,\label{length}\EE the module $V$ is called {\it cuspidal}.  In this case denoted by ${\rm b}(V)$ the minimal $p$ in (\ref{length}).
If ${\rm b}(V)=1$, the cuspidal module $V$ is called {\it a module of the intermediate
 series}.

A ${\frak g}$-module $V$ is called a {highest} (resp.
{lowest) weight module}, if there exists a nonzero $v
\in V_{\lambda}$ such that

1) $V$ is generated by $v$ as  ${\frak g}$-module with $L_0v=\lambda v$, $G_0^+v=0$, $H_0v=h'v$  and $Cv=cv$ for some $h,h',c\in\c$;

2) ${{\frak g}}_+v=0 $ (resp. ${{\frak g}}_- v=0 $), where ${{\frak g}}_+=\sum_{i>0}{{\frak g}}_i$, ${{\frak g}}_-=\sum_{i<0}{{\frak g}}_i$.
Clearly highest or lowest weight modules are quasi-finite.

Similarly, we can define a weight module,  a cuspidal module and   a module of the intermediate
 series for the Virasoro algebra $\Vir$ and the twisted Heisenberg-Virasoro algebra $\mathfrak{t}$.

For $a, b\in\c$, the $\Vir$-module of the intermediate series  ${\mathcal A}_{a,\; b}=\sum_{i\in\z}\c v_i$  was given  as follows (see \cite{KS}):
\BE  L_mv_i=(a+i+b m)v_{m+i}, \  \forall i, m\in\z.\label{inter-vir}\EE

It is well-known that ${\mathcal A}_{a,\; b}\cong{\mathcal A}_{a+1,\; b}$ for all $a, b\in\c$, then we can always suppose that $a\not\in\z$ or $a=0$ in ${\mathcal A}_{a,\; b}$.
Moreover, the module
${\mathcal A}_{a,\; b}$ is simple if $a\notin\z$ or $b\ne0, 1$.
In the opposite case,  ${\mathcal A}_{0,0}$ has $\c v_0$ as a submodule, and its corresponding
quotient is denoted by ${\mathcal A}_{0,0}'$. Dually,
${\mathcal A}_{0,1}$ has $\sum_{i\ne 0}\mathbb C v_i$ as a submodule, which is
isomorphic to ${\mathcal A}_{0,0}'$. For
convenience, we simply write ${\mathcal A}_{a,b}'={\mathcal A}_{a,b}$ when ${\mathcal A}_{a,b}$ is simple.

All simple weight $\Vir$-modules with finite-dimensional weight spaces  were  classified in \cite{Ma}.

\begin{theo}\label{vir}
Let $V$ be a simple weight $\Vir$-module with finite-dimensional weight spaces.
Then $V$ is a highest weight module, a lowest weight module, or a module of the intermediate series.
\end{theo}

 Extending this classification in a nontrivial way, L\"u and Zhao in \cite{LvZ} classified all simple weight $\mathfrak t$-modules with finite-dimensional weight spaces.

\begin{theo}\label{thm-twisted}
Let $V$ be a simple weight $\mathfrak t$-module with finite-dimensional weight spaces.
Then $V$ is a highest weight module, a lowest weight module, or a module of the intermediate series  ${\mathcal A}_{a,\; b,\; c}=\sum_{i\in\z}\c v_i$ for some $a, b, c\in\c$, with
\BE L_mv_i=(a+i+bm)v_{m+i},\quad H_mv_i=cv_{m+i},\quad \forall i, m\in\z.\label{inter-thv}\EE
Moreover, the module
${\mathcal A}_{a,\; b,\; c}$ is simple if $a\notin\z$ or $b\ne0, 1$ or $c\ne 0$.
For
convenience, we also use ${\mathcal A}_{a, b, c}'$ to denote the nontrivial simple sub-quotient of ${\mathcal A}_{a, b, c}$.
\end{theo}

\section{Cuspidal modules of the subalgebras ${\mathfrak q^\pm}$}

In this section, we classify all simple weight ${\mathfrak q^\pm}$-modules with finite-dimensional weight spaces.

\begin{lemm}\label{nilp}
Let $V=\sum V_i$ be a cuspidal $\mathfrak t$-module including no any trivial $\mathfrak t$-vector. There exists $m\in\z^+$ such that for any nonzero $v\in V_i$, $i\in\z^*$, any simple $\mathfrak t$-module $V'\subset U(\mathfrak t)v$ satisfies \BE V'\subset \sum_{1\le k\le m}\mathfrak t^kv.\label{maxq}\EE
\end{lemm}
\begin{proof}
From representation theory of the Virasoro algebra, $\dim\,V_j=N$ for some positive integer $N$ holds for almost all $j\in\z$ and $C$ acts on $V$ as zero (see \cite{KS, MP}).
Consider a composition series of $\mathfrak t$-submodules of $U(\mathfrak t)v$:
\begin{equation}0\subset V'=V^{(1)}\subset V^{(2)}\subset \cdots \subset V^{(p)}=U(\mathfrak t)v,\label{fil1}\end{equation}
where $V^{(1)}, \cdots ,V^{(p)}$ are $\mathfrak t$-submodules of $V$, $1\le p\le N$, and
the quotient modules $W^{(j)}=V^{(j)}/V^{(j-1)}=\sum\c \bar{v}_i^{(j)}$ are isomorphic to ${\mathcal A}_{a, b_j, c_j}'$ for some $a, b_j, c_j\in\c$ by Theorem \ref{thm-twisted}.

Assume that $0\ne v=a_1v_i^{(1)}+\cdots+a_{p}v_i^{(p)}\in V_i^{(p)}:=V_i\cap V^{(p)}$ for some $a_1, \cdots a_{p}\in\c$ such that $\{v_i^{(j)}+V^{(j-1)}, i\in\mathbb Z\}$ is a basis of $V^{(j)}/V^{(j-1)}$ for $j=1,\cdots, p$, where $V^{(0)}=0$. Then we have
\begin{eqnarray} &&L_jv\equiv a_{p}(a+b_{p}j+i)v_{i+j}^{(p)}\ ({\rm mod}\, V_{i+j}^{(p-1)}),\label{equiv1}\\
&& H_jv\equiv c_pv_{i+j}^{(p)}\ ({\rm mod}\, V_{i+j}^{(p-1)})\label{equiv2}\end{eqnarray} for any $j\in\z$.

Now we shall prove that there exists $0\ne u\in V_{i+k}^{(p-1)}$ for some $k\in\z$ such that \begin{equation}u\in\sum_{i=1}^2\mathfrak t^iv\label{ut2}\end{equation} case by case.

\noindent{\bf Case 0.}  $V^{(p)}/V^{(p-1)}$ is a trivial $\frak t$-module.  In this case $\frak tv\in V^{(p-1)}$. If $\frak tv=0$, then $v$ is a trivial $\frak t$-vector. It gets a contradiction. So $\frak tv\ne0$, and then \eqref{ut2} holds.

\noindent{\bf Case 1.} $c_p\ne 0$.  In this case, for any $k, k'\in\z$, $X, Y\in\{L, H\}$, there exists $d_{k, k'}\in\c$ such that $u=X_{k'}Y_{k-k'}v-d_{k, k'}H_kv\in V_{i+k}^{(p-1)}$ if $X_{k'}Y_{k-k'}v\ne 0$ by (\ref{equiv1}) and (\ref{equiv2}). If $u$ is always zero, then $$X_{k'}Y_{k-k'}v=d_{k, k'}H_kv, \ \forall X, Y\in\{L, H\},\ k, k'\in\z$$ for some $d_{k, k'}\in\c$. Then $\span_\c\{H_kv\mid k\in\z\}$ becomes a $\mathfrak t$-submodule and $p=1$.
So (\ref{ut2}) holds.

\noindent{\bf Case 2.} $c_p=0$. In this case, $H_jv\in V_{i+j}^{(p-1)}$ for any $j\in\z$. So there exists $u=H_kv\in V_{i+k}^{(p-1)}$ such that $0\ne u\in\mathfrak tv$  by (\ref{equiv2}) if $H_kv\ne 0$ for some $k\in\z$. In this case (\ref{ut2}) also holds.

\noindent{\bf Case 3.} Now we can suppose that $H_jv=0$ for any $j\in\z$, then $U(\mathfrak t)v=U(\Vir)v$.

\noindent{\bf Subcase 3.1.} $L_jv\not\in V_{i+j}^{(p-1)}, \forall j\in\z$.   For any $k'\in\z$, there exists $d\in \c$ such that $u=L_{k'}L_{k-k'}v-dL_kv\in V_{i+k}^{(p-1)}$ if $L_{k'}L_{k-k'}v\ne 0$ by (\ref{equiv1}). Similarly, if the $u$ is always zero, then $\span_\c\{L_kv\mid k\in\z\}$ becomes a $\Vir$-submodule and then $p=1$. So (\ref{ut2}) still holds.

\noindent{\bf Subcase 3.2.} $u=L_kv\in V_{i+k}^{(p-1)}$ for some $k\in\z$.  In this case (\ref{ut2}) holds if $L_kv\ne 0$.

\noindent{\bf Subcase 3.3.} The only left case is that $L_{k_0}v=0$ for some $k_0\in\z$. In this case $L_jv\not\in V_{i+j}^{(p-1)}$ for any $j\ne k_0$ by (\ref{equiv1}).
As in Subcase 3.1, for any $k\ne k_0, k'\in\z$, there exists $d_{k, k'}\in \c$ such that $u=L_{k'}L_{k-k'}v-d_{k, k'}L_kv\in V_{i+k}^{(p-1)}$ if $L_{k'}L_{k-k'}v\ne 0$.

Suppose that the above $u$ is always zero, then $$L_{k'}L_{k-k'}v=d_{k, k'}L_kv$$ for some $d_{k, k'}\in\c$ if $k\ne k_0$. Moreover, for any $k\ne k_0, k\ne 0, j\ne0$, $(k-2j)L_kL_{k_0-k}v=[L_j, L_{k-j}]L_{k_0-k}v=L_jL_{k-j}L_{k_0-k}v-L_{k-j}L_jL_{k_0-k}v=d_{k-j, k_0-k}L_jL_{k_0-j}v-d_{k-j, j}L_kL_{k_0-k}v$.  So we have $$L_jL_{k_0-j}v=\frac{(k-2j)+d_{k-j, j}}{d_{k-j, k_0-k}}L_kL_{k_0-k}v,$$ where $d_{k-j, k_0-k}\ne 0$ since $L_jv\not\in V_{i+j}^{(p-1)}$ for any $j\ne k_0$.
In this case $\sum_{j\in\z, j\ne k_0}\c L_jv+\c u_{k_0}$ becomes a Vir-submodule, where $u_{k_0}=L_kL_{k_0-k}v\ne 0$ for some $k\in\z^*$. Then $p=1$.

Now (\ref{ut2}) holds in all cases. If $k+i=0$, we can replace $u$ by $L_ju$ or $H_ju$ for some $j\ne 0$ and get that there exists $0\ne u\in V_{i+k}^{(p-1)}$ for some $k\in\z$ with $k+i\ne 0$ such that \begin{equation}u\in\sum_{i=1}^3\mathfrak t^iv.\label{ut3}\end{equation}
By induction we can find $v_j^{(1)}\in \sum_{1\le k\le 3p}\mathfrak t^kv$ for some $j\in\z$ and then \begin{equation} V'\subset \sum_{1\le k\le 3p+2}\mathfrak t^kv\label{maxq}\end{equation} since $L_{k-j}v_j^{(1)}=(a+b_1(k-j)+j)v_k^{(1)}$ and $H_{k-j}v_j^{(1)}=c_1v_k^{(1)}$ for any $k\in\z$. Set $m=3N+2\ge 3p+2$, the lemma follows.
\end{proof}

 Set ${\mathcal O}_n:=\{0, \pm1, \pm2, \cdots, \pm n\}$ for $n\in\z_+$. We denote $\mathfrak q:=\mathfrak q^+ (\cong \mathfrak q^-)$, $G_j:=G_j^+$ for all $j\in\z$ and $\mathfrak q_{\bar1}=\span_\c\{G_j\mid j\in\z\}$.

\begin{lemm}\label{induction} Let $V$ be a $\mathfrak q$-module.
Suppose that there exist $n , m\in\z_+$ and $n>m$ such that \BE G_{r_1}G_{r_2}\cdots G_{r_m}V=0\label{nilp-0} \EE for all $ r_1, r_2, \cdots r_m\in {\mathcal O}_n$, then $(\ref{nilp-0})$ holds for all  $r_1, r_2, \cdots r_m\in \z$.
\end{lemm}

\begin{proof}
We shall use the induction on $n$ and first prove that $$G_{r_1}\cdots G_{r_{m-1}}G_{n+1}V=0$$ for all  $r_1, r_2,\cdots, r_{m-1}\in {\mathcal O}_n$.

By action of $L_{n+1-r_m}$ on (\ref{nilp-0}) (where $r_1<r_2<\cdots<r_m$) we get
\BE G_{r_1}G_{r_2}\cdots G_{r_{m-1}}G_{n+1}V=0\label {nilp-1}\EE for all $-n\le r_1<r_2<\cdots <r_{m-1}\le n-1$.
The left is to prove that $$G_{r_1} G_{r_2}\cdots G_{r_{m-2}} G_n G_{n+1} V = 0$$ for all  $r_1, r_2, \cdots r_{m-2}\in {\mathcal O}_{n}$.

By actions of $L_1$ and $H_1$ on (\ref{nilp-1}) with $r_{m-1}=n-1$, we get
\begin{eqnarray} && ((n-\frac32)G_{r_1}\cdots G_{r_{m-2}}G_{n}G_{n+1}+(n+\frac12)G_{r_1}\cdots G_{r_{m-2}}G_{n-1}G_{n+2})V=0,\label {nilp-3}\\
&&(G_{r_1}\cdots G_{r_{m-2}}G_{n}G_{n+1}+G_{r_1}\cdots G_{r_{m-2}}G_{n-1}G_{n+2})V=0.\label {nilp-4}
\end{eqnarray}

Combining with (\ref{nilp-3}) and (\ref{nilp-4}), we get
\BE G_{r_1}\cdots G_{r_{m-2}}G_{n}G_{n+1}V=0 \label {nilp-5} \EE
for all $-n\le r_1<\cdots <r_{m-2}\le n-2$.

In the same way, by the actions of $L_1$ and $H_1$ on (\ref{nilp-5}) with $r_{m-2}=n-2$, we get
\begin{eqnarray} && G_{r_1}\cdots G_{r_{m-3}}G_{n-1}G_{n}G_{n+1}V=0 \label {nilp-10}
\end{eqnarray}
for all $-n\le r_1<r_2<\cdots <r_{m-3}\le n-3$.

Repeating the above steps by setting $r_{m-3}=n-3, \cdots, r_1=n-m+1$, respectively, we can get $G_{r_1}G_{r_2}\cdots G_{r_{m-1}}G_{n+1}V=0$ for all  $r_1, r_2, \cdots r_{m-1}\in \mathcal O_{n}$.

Similarly we can get $G_{-n-1}G_{r_2}G_{r_3}\cdots G_{r_m}V=0$ for all  $r_2, r_3, \cdots r_m\in \mathcal O_{n}$.
Then (\ref{nilp-0}) holds for all $r_1, r_2, \cdots r_m\in {\mathcal O}_{n+1}$. The lemma follows by induction on $n$.
\end{proof}

\begin{prop}\label{pssn0}
Let $V=\sum V_i$ be a simple cuspidal
${\mathfrak q}$-module. Then $V$ is isomorphic to a ${\mathfrak q}$-module of the intermediate series: $V=\sum\mathbb C v_i\cong {\mathcal A}_{a, b, c}'$ for some $a, b, c\in\c$ with
$$L_mv_i=(a+i+bm)v_{m+i},\quad H_jv_i=cv_{i+j},\quad G_jV=0$$
for all $m,i, j\in\z$.
\end{prop}

\begin{proof}
Clearly there exists  some positive integer $N$ such that  $\dim\,V_i\le N$ holds for almost $i\in\z$ and $C$ acts on $V$ as zero (see \cite{KS, MP}). Now ${\mathfrak q}_{\bar1}^iV$ is ${\mathfrak q}$-submodule since ${\mathfrak q}_{\bar1}^{i+1}V\subset {\mathfrak q}_{\bar1}^iV$ for all $i\in\n$. So ${\mathfrak q}_{\bar1}V=V$ or ${\mathfrak q}_{\bar1}V=0$.

If there exists a trivial $\mathfrak t$-vector $v\in V$, then $V=U({\mathfrak q}_{\bar1})v$. Clearly $V=\c v+\sum_{i\ge 1}{\mathfrak q}_{\bar1}^iv$. Moreover we have ${\mathfrak q}_{\bar1}^iv=\{v\in V\mid H_0v=iv\}$ for any $i\in\z_+$. It is easy to see that ${\mathfrak q}_{\bar1}v=0$. Otherwise $\sum_{i\ge 1}{\mathfrak q}_{\bar1}^iv$ is a proper $\mathfrak q$-submodule of $V$.
So we can suppose that $V$ does not contain any trivial $\mathfrak t$-vectors.

\noindent {\bf Claim } For any $n\in\z_+$, there exists  $0\not=v\in V_i$ for some $i\in \z^*$ such that $G_rv=0$ for all $r\in {\mathcal O}_n$.

In fact, if $G_rv\not=0$ for some $r\in {\mathcal O}_n$, we may replace $v$ by $G_rv\in V_{r+i}$. So this claim holds for any $n\in\z_+$.

Choose $n\gg N$ for the $v, n$ in the Claim. Clearly $\mathfrak tv\ne 0$ (otherwise $\c v$ is a trivial $\mathfrak t$-module).
By Theorem \ref{thm-twisted}, we can choose a simple $\mathfrak t$-module $V'=\sum \c v_i\cong {\mathcal A}_{a, b, c}'\subset U(\mathfrak t)v$ for some $a, b, c\in\c$, and then $V=U({\mathfrak q}_{\bar1})V'$. By Lemma \ref{nilp}, there exists $m\in\z^+$ with $m<n$ such that \BE V'\subset \sum_{k\le m}\mathfrak t^kv.\label{maxq}\EE
Now for all $ r_1, r_2, \cdots r_m\in {\mathcal O}_n$, we have
\begin{eqnarray*} G_{r_1}G_{r_2}\cdots G_{r_m}V'\subseteq\bigoplus_{j=1}^m x_j G_{r_j}v=0,\end{eqnarray*}
where the $x_j$'s are some elements in $U({\mathfrak q})$.
It follows from $V=U({\mathfrak q}_{\bar1})V'$ that
$$
G_{r_1}G_{r_2}\cdots G_{r_m}V=0,\quad \forall r_1, r_2, \cdots r_m\in {\mathcal O}_n.
$$
By Lemma \ref{induction}, we get
\begin{equation}{\mathfrak q}_{\bar1}^mV=0. \label{grg01}\end{equation}
If ${\mathfrak q}_{\bar1}V=V$ then ${\mathfrak q}_{\bar1}^mV=V=0$, which is a contradiction.
So ${\mathfrak q}_{\bar1}V=0$ and the proposition follows from Theorem \ref{thm-twisted}.
\end{proof}

\begin{rema} Recently, Proposition \ref{pssn0} was proved in \cite{CLW}  by the different methods.
\end{rema}

\section{Classification of simple cuspidal
$\mathfrak g$-modules}

In this section we shall classify all simple cuspidal
modules for the $N=2$ Ramond algebra ${\frak g}$.

\begin{prop}\label{pn=2}
Let $V$ be a simple cuspidal
${\frak g}$-module. Then ${\rm b}(V)\le 2$.
\end{prop}
\begin{proof}
Clearly $C$ acts on $V$ as zero.

Now the subalgebras ${\mathfrak q}^\pm=\span_\c\{L_m, H_m, G_m^\pm, C\mid m\in\z\}$ are all isomorphic to the Lie superalgebra $\mathfrak q$ in Section 3. The even part $\frak g_{\bar0}$ is just the twisted Heisenberg-Virasoro algebra $\mathfrak t$ in Section 2.2.

If there exists a trivial $\mathfrak q^+$-vector $v\in V$, then $V=U({\mathfrak q}_{\bar1}^-)v$. Clearly $V=\c v+\sum_{i\ge 1}({\mathfrak q}_{\bar1}^-)^iv$ since $({\mathfrak q}_{\bar1}^-)^iv=\{v\in V\mid H_0v=iv\}$ for any $i\in\z_+$. It is easy to see that ${\mathfrak q}_{\bar1}^-v=0$. Otherwise $\sum_{i\ge 1}({\mathfrak q}_{\bar1}^-)^iv$ is a proper $\mathfrak g$-submodule of $V$.
So we can suppose that $V$ does not contain any trivial $\mathfrak q^\pm$-vector.

By Proposition \ref{pssn0}, we can choose a simple ${\mathfrak q}^+$-module $U^+=\sum\c u_i^+$ of $V$ such that $G_r^+U^+=0$ for all $r\in\z$. In this case $V=U(G^-)U^+$, where $G^-=\span_\c\{G_r^-\mid r\in\z\}$ is a subalgebra of $\frak g$.

Now we can suppose that $G^-U^+\ne 0$ (otherwise $V$ is a trivial $\frak g$-module). Set $(G^-)^0U^+=U^+$ and $(G^-)^{i+1}U^+=G^-(G^-)^{i}U^+$ for all $i\ge 0$.
Then we have $G^+(G^-)^{i+1}U^+\subset (G^-)^{i}U^+$. Moreover
\BE V=\sum_{i\ge 0}(G^-)^{i}U^+.\label{v000}  \EE

By Proposition \ref{pssn0} we can choose a simple ${\mathfrak q}^-$-module $U^-$ of $V$ such that $U^-=\sum\c u_j^-\cong{\mathcal A}_{a, b, c}'$ with $L_mu_j^-=(a+bm+j)u_{m+j}^-$, $H_mu_j^-=cu_{m+j}^-$ and $G_{m}^-u_j^-=0$  for any $m, j\in\z$, for some $a, b, c\in\c$.

Suppose that \BE u_i^-\in (G^-)^nU^+\label{neq001}\EE  for some $i\in\z$ and $n\in\z_+$.
By the actions of suitable $L_m$'s or $H_m$'s on $(\ref{neq001})$ we get $U^-\subset (G^-)^nU^+$.

In this case,
\BE V=\sum_{i=0}^n(G^+)^{i}U^-=\sum_{j=0}^n(G^-)^{j}U^+\label{v000}  \EE
and $(G^+)^{n}U^-\subset(G^+)^n(G^-)^{n}U^+\subset U^+$ (see the diagram below).
Moreover $(G^+)^{n}U^-\ne 0$ (in fact, if $(G^+)^{n}U^-=0$, replace $U^+$ by $U^-$ and repeate the above step, one has $(G^-)^{n}U^+=0$).
By the irreducibility of $U^+$ as $\mathfrak t$-module, we get \BE (G^+)^{n}U^-=U^+\label{mp=0}. \EE
So we have the following diagram:

\centerline{\xymatrix{
U^+\ar@{->}[r]\ar@{=}[d]^\cup&G^-U^+\ar@{->}[r]\ar@{<-}[d]^\cup&\cdots\ar@{->}[r]&(G^-)^{n-1}U^+\ar@{->}[r]\ar@{<-}[d]^\cup&(G^-)^{n}U^+\ar@{<-}[d]^\cup\\
(G^+)^{n}U^-\ar@{<-}[r]&(G^+)^{n-1}U^-\ar@{<-}[r]&\cdots\ar@{<-}[r]&G^+U^-\ar@{<-}[r]&U^-}}

Moreover from the diagram we see that $(G^-)^{n}U^+=U^-$ and $(G^-)^{i}U^+=(G^+)^{n-i}U^-$ for all $0\le i\le n$. In fact,
 from $U^-\subset (G^-)^{n}U^+$, we can get $(G^+)^{n}U^-\subset U^+$.  Due to that $U^+$ is irreducible,  $U^+=(G^-)^{n}U^-$. So
$(G^-)^{i}U^+\subset (G^+)^{n-i}U^-$ by actions of suitable $G^-$.

Now by Lemma \ref{lemmn=3} below we see that $n\le 2$ and get the proposition.

If $n=1$, then $V=U^++U^-$ and $G^-U^+=U^-$ by (\ref{mp=0}).

If $n=2$, then $V=U^++G^-U^++U^-$ and $G^-U^+=G^+U^-$ by (\ref{mp=0}).
\end{proof}

For convenience, set $K_{r,s}=[G_r^+, G_s^-]=-2L_{r+s}+(r-s)H_{r+s}$, then
we have the following equalities:
\begin{eqnarray} && K_{r,s}u_i^-=(-2(a+b(r+s)+i)+(r-s)c)u_{r+s+i}^-,\label{kl1}\\
  && [K_{r,s}, G_t^+]=2(r-t)G_{r+s+t}^+.\label{kg1}
\end{eqnarray}

\begin{lemm} \label{lemmn=2}
For any $r_1, r_2, s_1, s_2\in\z$, \BE G_{r_1}^-G_{r_2}^-G_{s_1}^+G_{s_2}^+u_i^-=d(s_2-s_1)(r_1-r_2)u_{i+r+s+s_1+s_2}^-,\label{gg2}\EE
where $d=(2b+c)(2+c-2b)$.
\end{lemm}
\begin{proof}
By (\ref{kl1}) and (\ref{kg1}) we have
\begin{eqnarray*}&&G_{r_1}^-G_{r_2}^-G_{s_1}^+G_{s_2}^+u_i^-\\
&=&G_{r_1}^-K_{s_1, r_2}G_{s_2}^+u_i^--G_{r_1}^-G_{s_1}^+G_{r_2}^-G_{s_2}^+u_i^-\\
&=&G_{r_1}^-G_{s_2}^+K_{s_1, r_2}u_i^-+2(s_1-s_2)G_{r_1}^-G_{s_2+s_1+r_2}^+u_i^--G_{r_1}^-G_{s_1}^+K_{s_2, r_2}u_i^-\\
&=&K_{s_2, r_1}K_{s_1, r_2}u_i^--K_{s_1, r_1}K_{s_2, r_2}u_i^-+2(s_1-s_2)K_{s_1+s_2+r_2, r_1}u_i^-\\
&=&(-2(a+b(s_1+r_2)+i)+(s_1-r_2)c)(-2(a+b(s_2+r_1)+i+s_1+r_2)+(s_2-r_1)c)u^-\\
&&-(-2(a+b(s_2+r_2)+i)+(s_2-r_2)c)(-2(a+b(s_1+r_1)+i+s_2+r_2)+(s_1-r_1)c)u^-\\
&&+2(s_1-s_2)(-2(a+b(r_1+r_2+s_2+s_2)+i)+(s_1+s_2+r_2-r_1)c)u^-\\
&=&(2+c-2b)(2b+c)(s_2-s_1)(r_1-r_2)u^-,
\end{eqnarray*}where $u^-=u_{i+r_1+r_2+s_1+s_2}^-$.
\end{proof}

Now we shall derive a key identity.

\begin{lemm} \label{lemmn=3}
For any $r_1, r_2, r_3, s_1, s_2, s_3\in\z$, \BE G_{r_1}^-G_{r_2}^-G_{r_3}^-G_{s_1}^+G_{s_2}^+G_{s_3}^+U^-=0.\label{gg3}\EE
\end{lemm}
\begin{proof}
\begin{eqnarray*}&&G_{r_1}^-G_{r_2}^-G_{r_3}^-G_{s_1}^+G_{s_2}^+G_{s_3}^+u_i^-\\
&=&G_{r_1}^-G_{r_2}^-K_{s_1, r_3}G_{s_2}^+G_{s_3}^+u_i^--G_{r_1}^-G_{r_2}^-G_{s_1}^+G_{r_3}^-G_{s_2}^+G_{s_3}^+u_i^-\\
&=&2(s_1-s_2)G_{r_1}^-G_{r_2}^-G_{s_1+s_2+r_3}^+G_{s_3}^+u_i^-+2(s_1-s_3)G_{r_1}^-G_{r_2}^-G_{s_2}^+G_{s_1+s_3+r_3}^+u_i^-\\
&&+G_{r_1}^-G_{r_2}^-G_{s_2}^+G_{s_3}^+K_{s_1, r_3}u_i^--G_{r_1}^-G_{r_2}^-G_{s_1}^+K_{s_2, r_3}G_{s_3}^+u_i^-+G_{r_1}^-G_{r_2}^-G_{s_1}^+G_{s_2}^+G_{r_3}^-G_{s_3}^+u_i^-\\
&=&2(s_1-s_2)G_{r_1}^-G_{r_2}^-G_{s_1+s_2+r_3}^+G_{s_3}^+u_i^-+2(s_1-s_3)G_{r_1}^-G_{r_2}^-G_{s_2}^+G_{s_1+s_3+r_3}^+u_i^-\\
&&+G_{r_1}^-G_{r_2}^-G_{s_2}^+G_{s_3}^+K_{s_1, r_3}u_i^--2(s_2-s_3)G_{r_1}^-G_{r_2}^-G_{s_1}^+G_{s_3+s_2+r_3}^+u_i^-\\
&&-G_{r_1}^-G_{r_2}^-G_{s_1}^+G_{s_3}^+K_{s_2, r_3}u_i^-+G_{r_1}^-G_{r_2}^-G_{s_1}^+G_{s_2}^+K_{s_3, r_3}u_i^-\\
&=&2d(s_1-s_2)({r_1}-{r_2})(s_3-s_1-s_2-r_3)u^-+2d(s_1-s_3)(r_1-r_2)(s_1+s_3+r_3-s_2)u^-\\
&&-d(2a+(2b-c)s_1+(2b+c)r_3+2i)(r_1-r_2)(s_3-s_2)u^-\\
&&-2d(s_2-s_3)(r_1-r_2)(s_2+s_3+r_3-s_1)u^-\\
&&+d(2a+(2b-c)s_2+(2b+c)r_3+2i)(r_1-r_2)(s_3-s_1)u^-\\
&&-d(2a+(2b-c)s_3+(2b+c)r_3+2i)(r_1-r_2)(s_2-s_1)u^-=0
\end{eqnarray*}for any $i\in\z$, where $u^-=u_{i+r_1+r_2+r_3+s_1+s_2+s_3}^-$.
\end{proof}

Next we shall determine the precise structure of the simple cuspidal $\mathfrak g$-modules.

\begin{prop}

Let $V$ be a simple cuspidal $\mathfrak g$-module of the intermediate series, then $V$ is isomorphic to the simple sub-quotient of ${\mathcal R}_{a, b}$ (up to parity-change) for some $a, b\in\c$, where ${\mathcal R}_{a, b}:=\sum\c v_i^++\sum\c v_i^-$  with
\begin{eqnarray}
&&L_mv_i^+=(a+i+bm)v_{m+i}^+,\\
&& L_mv_i^-=(a+i+(b-\frac{1}{2})m)v_{m+i}^-,\\
&&H_mv_i^+=(2-2b)v_{m+i}^+,\\
&& H_mv_i^-=(1-2b)v_{m+i}^-,\\
&&G_m^-v_i^+=v_{m+i}^-,\\
&& G_m^+v_i^-=-(2a+(4b-2)m+2i)v_{m+i}^+
\end{eqnarray} for all $m, i\in\z$.

\end{prop}

\begin{proof}
In this case  $V=U^++U^-$ and $G^-U^+=U^-$ by Proposition \ref{pn=2}, where $U^\pm$ are simple $\mathfrak t$-modules.

Set  $U^+=\sum \c v_i^+\cong {\mathcal A}_{a, b, c}'$ for some $a, b, c\in\c$. Due to that $G^-U^+=U^-$, we can suppose that $U^-=\sum\c v_i^-\cong {\mathcal A}_{a, b^-, c^-}'$ for some $b^-, c^-\in\c$ and $G_r^-v_i^+=g(r, i)v_{r+i}^-$ for almost all $0\ne g(r, i)\in\c$.
By $[L_m, G_n^-]v_i^+=(n-\frac m2)G_{m+n}^-v_i^+$ and $[H_m, G_n^-]v_i^+=-G_{m+n}^-v_i^+$, we have
\begin{eqnarray}
&&\hskip-20pt (a+b^-m+n+i)g(n, i){-}(a+bm+i)g(n, m+i){=}(n-\frac m2)g(m+n, i), \label{example10}\\
&&c^-g(n, i)-cg(n, m+i)=-g(m+n, i),\ \forall m, n, i\in\z.\label{example11}
\end{eqnarray}
Let $m=0$ in (\ref{example11}), then we get $c^-=c-1$.

\noindent{\bf Case 1.} $c=0$.

In fact, if $c=0$, then $c^-=-1$. For any $n, i\in\z$,
by (\ref{example11}) we get $g(n, i)=g(i)$ for some $g(i)\in\c$.  So we can suppose that $G_n^+v_i^-=h(n, i)v_{n+i}^+$ for some $h(n, i)\in \c$. By $[H_m, G_n^+]v_i^-=G_{m+n}^+v_i^-$, we have $h(n,i)=h(n+i)$ for some $h(n)\in\c$ for any $n, i\in\z$. By
$$[G_n^+, G_m^-]v_i^-=-2L_{m+n}v_i^++(n-m)H_{m+n}v_i^-,$$
we get
$$h(n+i)g(n+i)=-2(a+b^-(m+n)+i)-(n-m)$$
for any $m, n, i\in\z$. Then $b^-=\frac12$. Similarly, by $[G_n^+, G_m^-]v_i^+=-2L_{m+n}v_i^++(n-m)H_{m+n}v_i^+$, we get $b=1$.
Then $g(n, i)=g(0,0)$ for all $n, i\in\z$ follows by (\ref{example10}).

\noindent{\bf Case 2.}  $c=1$.

 It is similar to prove that $g(n, i)=g(0,0)$ for all $n, i\in\z$ as in Case 1.

\noindent{\bf Case 3.} $c\ne 0, 1$.

Setting $i=0$ in (\ref{example11}) we get
\BE g(n, m)=\frac{c-1}cg(n, 0)+\frac1c g(m+n, 0),\ \forall m, n\in\z.\label{example12}\EE
Substituting (\ref{example12}) into (\ref{example11}) we get
\BE g(m+n+i, 0)=g(n+i, 0)+g(n+m, 0)-g(n, 0), \  \forall n, m, i\in\z.\label{example13}\EE
By (\ref{example13}) we get $g(m, 0)=g(0, 0)$ and then $g(m, n)=g(0, 0)$ for all $m, n\in\z$ by (\ref{example11}) again.

So in all cases we have $g(m, n)=g(0, 0)$ for all $m, n\in\z$.
Moreover, we have $b^-=b-\frac12$ by (\ref{example10}). By direct calculations,  we obtain
\begin{eqnarray*}
&&L_mv_i^+=(a+i+bm)v_{m+i}^+,\quad L_mv_i^-=(a+i+(b-\frac{1}{2})m)v_{m+i}^-,\nonumber \\
&&H_mv_i^+=(2-2b)v_{m+i}^+,\quad H_mv_i^-=(1-2b)v_{m+i}^-,\nonumber\\
&&G_m^-v_i^+=v_{m+i}^-,\quad G_m^+v_i^-=-(2a+(4b-2)m+2i)v_{m+i}^+\nonumber
\end{eqnarray*} for all $m, i\in\z$.
\end{proof}

\begin{rema}  ${\mathcal R}_{a, b}$ is not simple if and only if $a=0, b=1$ or $a=0, b=\frac12$. All indecomposable modules of of the intermediate series were given in \cite{FJS}.
\end{rema}

\begin{prop}\label{n=2structure}
Let $V$ be a simple cuspidal $\mathfrak g$-module with  ${\rm b}(V)=2$, then $V$ is isomorphic to ${\mathcal R}_{a, b, c}$ $($up to parity-change$)$ for some $a, b,c\in\c$ and $2b\pm c\ne 2$, where ${\mathcal R}_{a, b, c}:=(U+U^{+-})\oplus (U^++U^-)$ with
\begin{eqnarray}
&&L_mv_i=(a+i+bm)v_{m+i},\  \   L_mv_k^\pm=(a+k+(b-\frac{1}{2})m)v_{m+k}^\pm, \label{n=2L}\\
&&H_mv_i= c v_{m+i},\ \ H_mv_k^\pm=(c\pm1)v_{m+k}^\pm,\label{n=2hm}\\
&&G_r^\pm v_i=v_{r+i}^\pm,\ \  G_r^\pm v_i^\pm=0 \label{n=2gv},\\
&&L_mv_i^{+-}=(a+i+(b-1)m)v_{m+i}^{+-}+\frac12(2b-c-2)m^2v_{i+m},\label{n=2L+-}  \\
&& H_mv_i^{+-}=c v_{m+i}^{+-}-m(2b-c-2)v_{m+i}, \label{n=2H+-} \\
&&G_r^+v_k^-=v_{k+r}^{+-}+(c+2-2b)rv_{k+r},\label{n=2g+}\\
&&G_r^-v_k^+= -v_{r+k}^{+-}-\big(2a+2k+(2b+c)r\big)v_{k+r},\label{n=2g-}\\
&&G_r^-v_i^{+-}=-(2a+2i+(2b+c-2)r)v_{r+i}^-,\\
&&G_r^+v_i^{+-}=(2b-c-2)rv_{r+i}^+, \  \forall m, r, i\in\z.
\end{eqnarray}

\end{prop}

\begin{proof}

In this case $n=2$, $V=U^++G^-U^++U^-$ and $G^-U^+=G^+U^-$ by Proposition \ref{pn=2}.

First we can choose a simple $\mathfrak t$-module $U=\sum \c v_i\cong {\mathcal A}_{a, b, c}'$ of $G^-U^+$ for some $a, b, c\in\c$. Due to $U^\pm=G^\pm U$ are all simple $\mathfrak t$-modules, we can get $U^\pm=\sum\c v_i^{\pm}\cong {\mathcal A}_{a, b-\frac12, c\pm 1}'$ as in Case 1.
So we can obtain (\ref{n=2L}), (\ref{n=2hm}), (\ref{n=2gv}) and the coefficients of $v_{m+i}^{+-}$ in (\ref{n=2L+-}) and (\ref{n=2H+-}).

Notice that the quotient module $G^-U^+/U=\sum\c {\bar v}_i^{+-}$ (or $G^+U^-/U$) is also a simple $\mathfrak t$-module, and ismorphic to ${\mathcal A}_{a, b-1, c}'$ as in Case 1. So we can suppose that $$G_r^+v_k^-=v_{r+k}^{+-}+g(r, k)v_{r+k}$$ for some $g(r, k)\in\c$ and for any $r, k\in\z$.
Moreover, replace $v_k^{+-}+g(0, k)v_{k}$ by $v_k^{+-}$ one has $G_0^+v_k^-=v_{k}^{+-}$. It is \BE g(0, k)=0, \ \forall r, k\in\z. \label{eq0k0}\EE
By the action of $G_s^-$ on $G_r^+ v_i=v_{r+i}^+$ for any $r,s, i\in\z$, we can get
$$G_r^-v_k^+= -v_{r+k}^{+-}+h(r, k)v_{k+r}$$ for some $h(r, k)\in\c$.
By $[G_r^+, G_r^-]v_k=-2L_{2r}v_k$, one has
\BE g(r, k)+h(r, k)=-2(a+(2b-1)r+k), \ \forall r, k\in\z. \label{eq111} \EE
By $[H_m, G_r^-]v_k^+=- G_{m+r}^- v_k^+$ and $[H_m, G_r^+]v_k^-= G_{m+r}^+ v_k^-$, one has
\BE -f_1(m, r+k)+ch(r, k)-(c+1)h(r, m+k)=-h(m+r, k).\label{eq113}\EE
\BE f_1(m, r+k)+cg(r, k)-(c-1)g(r, m+k)=g(m+r, k).\label{eq115}\EE

Combining with (\ref{eq111}), (\ref{eq113}) and (\ref{eq115}) we get \BE g(m+r, k)-g(r, m+k)=(c+2-2b)m. \label{eq116}\EE
Setting $r=0$ and using (\ref{eq0k0}) and (\ref{eq111}), we get \BE g(m, k)=(c+2-2b)m, \ \ h(m, k)=-(2a+2k+(2b+c)m),\ \forall m, k\in\z. \label{eq116}\EE
So (\ref{n=2g+}) and (\ref{n=2g-}) hold. By actions of $L_m, H_m, G_r^\pm$ on (\ref{n=2g+}), we can easily get the remaining relations in the proposition. Moreover ${\mathcal R}_{a, b, c}$ is simple if and only if $2b\pm c\ne 2$.
\end{proof}

\begin{rema}

If $2b-c=2$, then $V=\span_\c\{v_i^-, v_i^{+-}\mid i\in\z\}$ is a simple module of the intermediate series, which is just isomorphic to the module ${\mathcal R}_{a, b}$ listed in Case 1 (up to parity-change).

If $2b+c=2$, then $V=\span_\c\{v_i^+, v_i^{+-}+2(a+i)v_i\mid i\in\z\}$ is a simple module of the intermediate series, which is also isomorphic to  ${\mathcal R}_{a, b}$ (up to parity-change).

If $2b-c=2b+c=2\, (b=1, c=0)$ and $a=0,
\frac12$, then $V$ is a trivial module.

\end{rema}

For convenience, we use ${\mathcal R}_{a, b, c}'$ to denote the simple sub-quotient of ${\mathcal R}_{a, b, c}$ in all cases $($including the module ${\mathcal R}_{a, b}$ in the case of $n=1)$.
${\mathcal R}_{a, b, c}$ (or ${\mathcal R}_{a, b, c}'$) is just corresponding to the finite conformal module given in \cite{CLam}.

\section{Quasi-finite modules over the $N=2$ Ramond algebra}
In this section,  we classify all simple weight modules with finite-dimensional weight spaces for the $N=2$ Ramond algebra.

\begin{theo} \label{pH1}
Let $V$ be a simple weight module with finite-dimensional weight spaces over $\frak g$. If $V$ is not a highest and lowest
module, then $V$ is cuspidal.
\end{theo}
\begin{proof} Suppose that $V=\sum_{k\in \z}V_k$ is a
simple quasi-finite ${\frak g}$-module without highest and
lowest weights. We shall prove that for any $n\in\z^*$, $k\in\z$,
\begin{eqnarray}\label{s===0}
L_n|_{V_k}\oplus L_{n+1}|_{V_k}\oplus H_{n+2}|_{V_k}\oplus G^+_{n}|_{V_k}\oplus G^-_{n+1}|_{V_k}: V_k\ \to\ V_{k+n}\oplus V_{k+n}\oplus V_{k+n+1}
\end{eqnarray}
is injective. In particular, by taking $n=-k$, we obtain that
$V$ is cuspidal.

In fact, suppose there exists some $v_0\in V_k$ such that
\begin{equation}\label{LLL-1111}L_nv_0=L_{n+1}v_0=G_{n}^+ v_0=G_{n+1}^- v_0=H_{n+2}v_0=0.\end{equation}
Without loss of generality, we can suppose $n>0$. Note that when
$\ell\gg0$, we have
$$\ell=n_1n+n_2(n+1)$$
for some $n_1, n_2\in\N$. From this and the relations in the
definition of $\frak g$, one can easily deduce that $L_\ell, G_{\ell}^\pm, H_\ell$ can be generated by $L_n,L_{n+1}, G_{n}^+, G_{n+1}^-, H_{n+2}$.
Therefore, there exists some $N_0>0$ such that
$$L_\ell v_0=G_{\ell}^\pm v_0=H_{\ell}v_0=0, \forall\ell\ge
N_0.$$
This means \begin{equation}{\frak g}_{[N_0, +\infty)}v_0=0,\label{v00}\end{equation} where ${\frak g}_{[N_0, +\infty)}=\oplus_{l\ge N_0} {\frak g}_l$.
Clearly $V=U(\frak g)v_0$ and $V_k=U(\frak g)_0v_0$. For any $v\in V_k$, there exists $u\in U(\frak g)_0$ such that $v=uv_0$.
Suppose that $$u=\sum a_{-i_1,..., -i_m, i_{m+1}, \cdots, i_n}X_{-i_1}\cdots Y_{-i_m}Z_{i_{m+1}}\cdots W_{i_n},$$ where $ X, \cdots,  Y,  Z \cdots, W \in \{L, H, G^\pm\}, i_1, \cdots, i_n>0$.
Define the negative degree $$n(v):={\rm max}\{i_1+\cdots+i_m\mid a_{-i_1,..., -i_m, i_{m+1}, \cdots, i_n}\ne 0\}.$$ It is well-defined since the sum is finite.
Choose a basis $\{v_1, v_2, \cdots, v_{d_k}\}$ of the finite-dimensional space $V_k$,
and set $n_k:={\rm max}\{n(v_{i})\mid 1\le i\le d_k\}$. Then we have \BE{\frak g}_{[N, +\infty)}V_k=0,\ {\rm where}\  N=N_0+n_k.\EE
Clearly ${\frak g}_+$ is generated by $S:=\{L_1, H_1, G_1^+, G_1^-\}$.
For any $i\ge1$,  set
 $U_{k+1}=SV_k=\sum_{x\in S}xV_k\subset V_{k+1}$. If $U_{k+1}=0$, then there exists a nonzero element $v\in V_k$ such that $Sv=0$ and then ${\frak g}_+v=0$. So $v$ is a highest weight vector, and $V$ is a highest weight module.

Now we can suppose that $U_{k+1}\ne 0$. For any $l\ge N$, we have
${\frak g}_lU_{k+1}=\sum_{x\in S}[{\frak g}_l, x]V_k+\sum_{x\in S}x{\frak g}_lV_k\subset {\frak g}_{l+1}V_k+S{\frak g}_lV_k=0$. Then we get
\BE{\frak g}_{[N, +\infty)}U_{k+1}=0.\EE
Repeating the above steps, we get $U_{k+i+1}=SU_{k+i}$ for any $i\ge 0$, and
\BE{\frak g}_{[N, +\infty)}U_{k+i}=0,\ \forall i\ge 0.\EE
For any $1\le j<N$, $z\in{\frak g}_{N-j}$, there exists $y\in {\frak g}_{N}$ such that $z=[y, L_{-j}]$. Then
$zU_{k+j}=[y, L_{-j}]U_{k+j}=yL_{-j}U_{k+j}-L_{-j}yU_{k+j}\subset yV_{k}-L_{-j}yU_{k+j}=0$.
So \BE{\frak g}_{[N-j, +\infty)}U_{k+j}=0,\ \forall 1\le j<N.\EE
Then ${\frak g}_+U_{k+N}=0$ for some $U_{k+N}\ne 0$.

Since the finite-dimensional subalgebra $\frak{a}:=\span_\c\{L_0, H_0, C\}$ is commutative, there exists a common eigenvector $w$ of $\frak{a}$ in $U_{k+N}$. Replace $w$ by $G_0^+w$ if $G_0^+w\ne 0$, $w$ becomes a highest weight vector of $\frak g$. This contradicts the assumption of the theorem.
\end{proof}

Combining with Theorem \ref{pH1} and Proposition \ref{pn=2},  we obtain our  main result of this paper.

\begin{theo} \label{main}
Let $V$ be a simple weight ${\frak g}$-module with finite-dimensional weight spaces. Then $V$ is a highest weight module, a lowest weight module, or ${\mathcal R}_{a, b, c}'$ for some $a, b, c\in\c$ $($up to parity-change$)$.
\end{theo}

\begin{rema}
This research was originally announced on April 2019 by arXiv:1904.08578. Recently \cite{XL} and \cite{BFIK} successively obtained a classification result for the Lie superalgebra $W(m, n)$ by a different method.

\end{rema}

\vskip30pt \centerline{\bf ACKNOWLEDGMENTS}

\vskip15pt We gratefully acknowledge the partial financial support from the NSFC (Nos. 12071405, 11971315, 12171155), and the Jiangsu Natural Science Foundation (No. BK20171294). Part of this work was done during the author's visiting the Chern Institute of Mathematics, Tianjin, China. The authors would like to thank the institute and Prof. Chengming Bai for their warm hospitality and support, and also deeply indebted to Prof. Rencai L\"u for helpful discussions.

\def\refname{

\end{document}